\documentclass[reqno,12pt]{amsart}
\usepackage{amsmath,amstext, amsthm, amssymb,epsfig,color,tikz,mathrsfs}
\usetikzlibrary{positioning}
\numberwithin{equation}{section}
\newtheorem{thm}{Theorem}[section]
\newtheorem{lem}[thm]{Lemma}

\newtheorem{conj}[thm]{Conjecture}

\newcommand{\be}{\begin{equation}}
\newcommand{\ee}{\end{equation}}
\newcommand{\ba}{\begin{array}}
\newcommand{\ea}{\end{array}}

\newcommand{\im}{\operatorname{Im}} 

\renewcommand{\em}{\it}

\newcommand{\bea}{\begin{eqnarray}}
\newcommand{\eea}{\end{eqnarray}}

\begin{document}
\title[On Cubic multisections of Eisenstein series]{On Cubic multisections of Eisenstein series}
\author{Andrew Alaniz and Tim Huber}
\address{Department of Mathematics, University of Texas - Pan American, 1201 West University Avenue,  Edinburg, Texas 78539,  USA}





\begin{abstract}
A systematic procedure for generating cubic multisections of Eisenstein series is given. The relevant series are determined from Fourier expansions for Eisenstein series by restricting the congruence class of the summation index modulo three. We prove that the resulting series are rational functions of $\eta(\tau)$ and $\eta(3\tau)$, where $\eta$ is the Dedekind eta function. A more general treatment of cubic dissection formulas is given by describing the dissection operators in terms of linear transformations. These operators exhibit properties that mirror those of similarly defined quintic operators.

\end{abstract}

\maketitle





\section{Introduction}
We begin with an identity observed by F. Garvan and communicated to the authors by B.C. Berndt. The purpose of this note is to place this identity in a larger context.
\begin{thm}
  \begin{align} \label{garvan}
    \sum_{n=0}^{\infty} \Bigl ( \sum_{d \mid 3 n +2} d \Bigr )q^{n} = \frac{3 (q^{3}; q^{3})_{\infty}^{6}}{(q;q)_{\infty}^{2}}, \qquad |q|<1. 
  \end{align}
\end{thm}
Here, we employ the usual notation for the $q$-expansion of the Dedekind eta function defined by $q^{-1/24}\eta(\tau) = \prod_{n=1}^{\infty} (1 - q^{n}) = (q;q)_{\infty}$, $q=e^{2\pi i \tau}$, $\im \tau > 0$. In this paper we prove Garvan's identity and place  \eqref{garvan} in the context of several infinite classes of relations for dissections of Eisenstein series. Some examples of formulas that ensue are
\begin{align*}
    \sum_{n=0}^{\infty} \left ( \sum_{d \mid 3n+1} d^{3} \right )q^{n} 
&= (q;q)_{\infty}^{8}+3^4 q\frac{(q^{3}; q^{3})_{\infty}^{12}}{(q;q)_{\infty}^{4}}, \\ 
    \sum_{n=0}^{\infty} \left ( \sum_{d \mid 3n+2} d^{7} \right )q^{n} = 3 \cdot 43(q;q)_{\infty}^{16} + 1&10 \cdot 3^{6} q (q;q)_{\infty}^{4} (q^{3}; q^{3})_{\infty}^{12} + 41 \cdot 3^{10}q^{2} \frac{(q^{3}; q^{3})_{\infty}^{24}}{(q;q)_{\infty}^{8}}, 
\end{align*}
and, if $ \left ( \frac{\cdot}{3} \right )$ denotes the Jacobi symbol modulo three,
\begin{align} \label{3core}
 \sum_{n=0}^{\infty} \left (\sum_{d\mid 3n+1} \left ( \frac{d}{3} \right ) \right ) q^{n}  &=  \frac{(q^{3}; q^{3})_{\infty}^{3}}{(q;q)_{\infty}}, \\ 
 \sum_{n=0}^{\infty} \left ( \sum_{d\mid 3n+2} \left ( \frac{d}{3} \right ) d^{4} \right )q^{n}   
= -15 (q;q)_{\infty}^{7}&(q^{3}; q^{3})_{\infty}^{3}- 3^{6} q\frac{(q^{3}; q^{3})_{\infty}^{15}}{(q;q)_{\infty}^{5}}. \nonumber
  \end{align}
The preceding relations are similar to dissections for
Eisenstein series of level five appearing in \cite{huber_jac}
involving Dirichlet characters modulo five. We have the
following dissection formulas for the quintic Dirichlet
character defined by $\langle \chi(n) \rangle_{n=0}^{4} =
  \langle 0, 1, -i, i, -1 \rangle$.
\begin{thm} \label{fjl}
\begin{align*}
  \sum_{n=0}^{\infty} \Bigl ( \sum_{d \mid 5 n +4} \chi(d) \Bigr ) q^{n} &=  \frac{i(q;q)_{\infty} (q^{5}; q^{5})_{\infty}}{(q^{2}; q^{5})_{\infty}^{3}(q^{3}; q^{5})_{\infty}^{3}}, \quad
  \sum_{n=0}^{\infty} \Bigl ( \sum_{d \mid 5 n +3} \chi(d) \Bigr ) q^{n} =  \frac{(1 + i) (q^{5}; q^{5})_{\infty}^{2}}{(q^{2}; q^{5})_{\infty}(q^{3}; q^{5})_{\infty}}, \\ 
  \sum_{n=0}^{\infty} \Bigl ( \sum_{d \mid 5 n +2} \chi(d) \Bigr ) q^{n} &= \frac{(1 - i)(q^{5}; q^{5})_{\infty}^{2}}{(q; q^{5})_{\infty}(q^{4}; q^{5})_{\infty}}, \quad 
  \sum_{n=0}^{\infty} \Bigl ( \sum_{d \mid 5 n +1} \chi(d) \Bigr ) q^{n} = \frac{(q;q)_{\infty} (q^{5}; q^{5})_{\infty}}{(q; q^{5})_{\infty}^{3}(q^{4}; q^{5})_{\infty}^{3}}.
\end{align*}
\end{thm}
The analogous dissection formulas we derive for Eisenstein of level three are consequences of a triple of fortuitous relations between generators for the relevant spaces
 of modular forms. To derive the expansions, we employ well known properties of the Borwein's cubic theta functions $a(q)$, $b(q)$, and $c(q)$, defined by
\begin{align}
a(q) &= \sum_{m,n =-\infty}^{\infty} q^{n^{2} + n m + m^{2}},  \quad
b(q) = \sum_{m,n =-\infty}^{\infty} \omega^{n-m} q^{n^{2} + n m + m^{2}}, \label{abq} \\ 
c(q) &= \sum_{m,n =-\infty}^{\infty}  q^{ \left ( n + \frac{1}{3} \right )^{2} + \left ( n + \frac{1}{3} \right ) \left ( m + \frac{1}{3} \right ) + \left ( m + \frac{1}{3} \right )^{2}}, \quad \omega = e^{2 \pi i/3}. \label{cq}
\end{align}

\begin{thm} \label{brv}
  \begin{align} \label{borwein}
  a^{3}(q) = b^{3}(q)+c^{3}(q)    
  \end{align}
  \begin{align}
a(q)	=a(q^{3})+2c(q^{3}), \qquad 
b(q)	=a(q^{3})-c(q^{3}).    
  \end{align} 
\end{thm}
Equivalent formulations of all three identities appear in Ramanujan's notebooks \cite{ramnote}. The first proof of \eqref{borwein} was given by J. Borwein and P. Borwein \cite{MR1010408}. Transformation formulas resulting from Theorem \ref{brv} appear in a different form in \cite{alt_base} and motivate cubic analogues of Jacobi's Principles of Duplication and Dimidiation \cite{funnov}.
Such formulas are subsumed in Ramanujan's ``theories of elliptic functions to alternative bases'' and termed Processes of Triplication and Trimidiation \cite{alt_base}.
The utility of these transformations in studying the coefficients of modular forms is apparent in light of the fact that $a(q)$ and $c(q)$ serve as generators for the principal congruence subgroup of level three \cite{MR1893493}. In fact, Sebbar \cite{MR1904094} proved that the principal congruence subgroups of level three and five are two of precisely six congruence subgroups of $SL(2, \Bbb Z)$ whose graded ring of modular forms is a polynomial ring with two generators, each of degree one.

The remainder of the paper is organized as follows. In Section 2, we
prove Garvan's identity \eqref{garvan} and generalize the dissection
technique to Eisenstein series of arbitrary weight and cubic
character. We identify classes of Eisenstein series whose cubic
dissections are expressible as rational functions of the Dedekind eta
function at argument $q$ and $q^{3}$. These representations
demonstrate general congruence properties for divisor sums and related
arithmetic functions that are stated at the beginning of Section 2 and
proven in the conclusion of Section 2. In Section 3 we provide
explicit matrix representations for dissection operators on certain
vector spaces. We use the operators to derive new congruence properties for the coefficients of corresponding eigenforms.

\section{Cubic dissections for Eisenstein series}
In this section, we demonstrate that dissections formulas like
\eqref{garvan} result from parameterizations for Eisenstein series in
terms of the cubic theta functions $a(q)$, $b(q)$, and $c(q)$. The
combinatorial consequences of our work include several new results for
twisted divisor sums and place a number of known results in
context. In particular, the ensuing dissection formulas exhibit the following combinatorial interpretations.
\begin{thm}\label{th:1} Suppose $\ell \in \Bbb N$ is odd and $3^{s}\mid \mid \ell$. Then for $k \in \Bbb N$, 
  \begin{align} \label{eq:la}
  \sum_{d \mid 3k+2} \left ( \frac{d}{3} \right )d^{2\ell} &\equiv 0
  \pmod{3^{s+1}}, \qquad   \sum_{d \mid 3k+1} \left ( \frac{d}{3}
  \right )d^{2\ell} \equiv c_{3}(k) \pmod{3^{s+1}}, \\  \sum_{d \mid
    3k + 2} d^{\ell} &\equiv 0 \pmod{3^{s+1}} , \qquad \sum_{d \mid
    3k + 1} d^{\ell} \equiv c_{3}(k) \pmod{3^{s+1}}, \label{eq:la1}
  \end{align}
where $c_{3}(k)$ is the number of $3$-core partitions of $k$; i.e., the number of partitions of $k$ satisfying the condition that no hook number in the Ferrers graph is divisible by $3$.
\end{thm}
If $\ell = s = 0$ in the latter congruence of \eqref{eq:la}, we obtain
an equality corresponding to \eqref{3core}, appearing in Granville and
Ono's work \cite{MR1321575}. The left congruences from \eqref{eq:la1} are proven
without reference to Eisenstein series in \cite{MR0242764}. The
remaining congruences from Theorem \ref{th:1} appear to be
new. At the end of this section, we show that each of the generating
functions for \eqref{eq:la}--\eqref{eq:la1} reduce modulo $3^{s+1}$ to dissected Eisenstein series of low weight. 
Moreover, our work results in explicit eta function expansions for the
generating functions corresponding to the trisected divisor sums in Theorem \ref{th:1}. The product
formulations follow from the product expansions \cite{MR1010408} for $b(q)$ and $c(q)$
\begin{align} \label{vb}
b(q) = \frac{(q;q)_{\infty}^{3}}{(q^{3};q^{3})_{\infty}}, \qquad  c(q) = 3 q^{1/3} \frac{(q^{3};q^{3})_{\infty}^{3}}{(q;q)_{\infty}}. 
\end{align}

To derive the aforementioned expansions for generating functions in terms
of cubic theta functions, two classes of Eisenstein series will be appropriate to our discussion. The first are normalized Eisenstein series of weight $k$ for the full modular group
\begin{align} \label{eis_full}
  E_{2k}(q) = 1 + \frac{2}{\zeta(1 - 2 k)} \sum_{n=1}^{\infty} \frac{n^{2k-1} q^{n}}{1 - q^{n}},
\end{align}
where $\zeta$ is the analytic continuation of the Riemann $\zeta$-function. We will also refer to the Hecke Eisenstein series associated with the Dirichlet character $\chi$ modulo three
\begin{align} \label{eisdef}
  E_{k,\chi}(q) = 1 +  \frac{2}{L(1 - k, \chi)} \sum_{n=1}^{\infty}
  \chi(n) \frac{n^{k-1} q^{n}}{1 - q^{n}}, 
\end{align}
where $L(1 - k, \chi)$ denotes associated Dirichlet $L$-series. 
Since we build cubic Eisenstein expansions inductively, we require representations
for low weight series from \cite{MR2249503,shen3}
\begin{align*}
  E_{4}(q) = a^{4}(q) + 8 ac^{3}(q), \ \ E_{6}(q) = a^{6}(q) - 20
    a^{3}(q) c^{3}(q) - 8 c^{6}(q), \ \ E_{3, \left (
        \frac{\cdot}{3} \right )}(q) = b^{3}(q).
\end{align*}

In the proof of the next Theorem, we show that Garvan's formula \eqref{garvan} arises from a dissection of the Hecke Eisenstein series of weight two and trivial character $\mathbf{1}$.
The technique applied here is representative of the methods used in the rest of the paper.
\begin{thm} \label{garvans} Let $\left ( \frac{\cdot}{3} \right )$ denote the Jacobi symbol modulo three. Then 
  \begin{align} \label{gn1}
    \sum_{n=0}^{\infty} \Bigl ( \sum_{d \mid 3 n +2} d \Bigr )q^{n} &= \frac{3 (q^{3}; q^{3})_{\infty}^{6}}{(q;q)_{\infty}^{2}}, \\ 
  \sum_{n=0}^{\infty}\left(\sum_{d\mid3n+1} \left ( \frac{d}{3} \right ) \right)q^{n} =
\frac{(q^{3}; q^{3})_{\infty}^{3}}{(q;q)_{\infty}},
 \qquad &\hbox{and} \qquad  
    \sum_{n=0}^{\infty}\left(\sum_{d\mid3n+2} \left ( \frac{d}{3} \right ) \right)q^{n} = 0. \label{gn2}
  \end{align}
\end{thm}

\begin{proof}
To derive \eqref{gn1}, we start with the cubic parameterization for the Eisenstein series of weight two \cite[Theorem 11.10]{MR2249503}
  \begin{align} \label{gr}
    \frac{3}{2}E_{2}(q^{3}) - \frac{1}{2} E_{2}(q) = a^{2}(q).
  \end{align}
Apply Theorem \ref{brv} to \eqref{gr} to derive
\begin{align} \label{su}
  \frac{3}{2}E_{2}(q) - \frac{1}{2} E_{2}(q^{1/3}) = a^{2}(q) + 4 a(q) c(q) + 4c^{2}(q).
\end{align}
Because of the unique determination of the Maclaurin expansions of each side of \eqref{su}, we may equate terms on each side with indices congruent to $2$ modulo
$3$ to derive
\begin{align} \label{vg}
  - 12  \sum_{n=0}^{\infty}\Bigl ( \sum_{d \mid 3 n +2} d \Bigr )q^{(3n+2)/3} = 4 c^{2}(q) = 36 q^{2/3} \frac{(q^{3}; q^{3})_{\infty}^{6}}{(q;q)_{\infty}^{2}}.
\end{align}
The product representations appearing in last equality follow from \eqref{vb}. 
A proof of \eqref{gn2} is likewise obtained by applying Theorem \ref{brv} to the identity \cite[Equation (2.33)]{MR2249503}
  \begin{align} \label{sha}
    E_{1, \left ( \frac{\cdot}{3} \right )}(q) = 1 + 6\sum_{n=1}^{\infty}  \left ( \frac{n}{3} \right )\frac{q^{n}}{1 - q^{n}} &= a(q).  \qedhere
  \end{align} 
\end{proof}
Since $a(q) \equiv 1\ (mod\ 3)$, we may equate terms on both sides of
\eqref{su} of index congrruent to one modulo three to obtain the
$\ell = 1$ case of the latter congruence in \eqref{eq:la1}.  By applying the
triple product identity to the theta expansion for $a(q)$ \cite[Lemma 2.1]{MR1243610}
\begin{align}
  a(q) = \theta_{3}(q) \theta_{3}(q^{3}) + \theta_{2}(q)\theta_{2}(q^{3}), \quad \theta_{2}(q) = \sum_{n=-\infty}^{\infty} q^{(n+\frac{1}{2})^{2}}, \quad \theta_{3} = \sum_{n=-\infty}^{\infty} q^{n^{2}},
\end{align}
we obtain a companion expansion to \eqref{gn1}. Most product
formulations for dissections involving the function $a(q)$ are
similarly unwieldy and will not be further studied here. 
\begin{thm}
\begin{align*}
  \sum_{n=0}^{\infty} \Bigl ( \sum_{d \mid 3 n +1} d \Bigr )q^{n} =& \frac{(-q;q^{2})_{\infty}^{2} (-q^{3}; q^{6})_{\infty}^{2}(q^{3}; q^{3})_{\infty}^{3}(q^{2}; q^{2})_{\infty} (q^{6}; q^{6})_{\infty}}{(q;q)_{\infty}} \\ & \qquad +\frac{4 q (q^{4}; q^{4})_{\infty} (q^{12}; q^{12})_{\infty} (q^{3}; q^{3})_{\infty}^{3}}{(q^{2}; q^{4})_{\infty} (q^{6}; q^{12})_{\infty} (q;q)_{\infty} }.
\end{align*}
\end{thm}
In order to achieve the claimed rational expansions in $\eta(q)$ and
$\eta(q^{3})$ for dissections of higher weight Eisenstein series,
we characterize
Eisenstein series of level three whose trisections lie in the
complex span of homogeneous polynomials in $c(q)$ and $b(q)$.
\begin{lem}
  If $\overline{n}$ denotes the least positive residue class of $n$ modulo $3$, then the following series are, respectively, homogeneous polynomials in $c(q)$ and $b(q)$ over $\Bbb Q$: 
  \begin{align*}
& \sum_{k=0}^{\infty} \left ( \sum_{d \mid 3k+ \overline{2n}} d^{2n-1} \right )q^{(3k+ \overline{2n})/3},  \qquad n \not\equiv 0 \pmod{3}, \\ \sum_{k=0}^{\infty} & \left ( \sum_{d \mid 3k+ \overline{2n+1}} \left ( \frac{d}{3} \right )d^{2n} \right )q^{(3k+ \overline{2n+1})/3}, \qquad  n \not\equiv 1 \pmod{3}.
  \end{align*} 
\end{lem}

\begin{proof}
  An induction argument along with recursion formulas from Lemma \ref{cooper_rec} imply 
  \begin{align} \label{itc1}
E_{2n+1,\left ( \frac{\cdot}{3} \right )}(q) = f_{2n+1} \Bigl (a(q), b(q) \Bigr ) \qquad \hbox{and} \qquad  E_{2n,\mathbf{1}}(q) = g_{2n}\Bigl (a(q), b(q) \Bigr )    
  \end{align}
for homogeneous polynomials $f_{2n+1}$ and $g_{2n}$ in $a(q)$ and $b(q)$ of degrees at most $2n+1$ and $2n$, respectively. By replacing $q$ by $q^{1/3}$ and applying Theorem  \ref{brv}, we derive expansions for $f_{2n+1} (a(q^{1/3}), b(q^{1/3}) )$ and $ g_{2n} (a(q^{1/3}), b(q^{1/3}) )$ as homogeneous polynomials in $a(q)$ and $c(q)$ with rational coefficients.
From \eqref{borwein}, and since we have assumed $2n+1 \not \equiv 0 \pmod{3},$ the $\overline{2n+1}$-dissection of $f_{2n+1} (a(q^{1/3}), b(q^{1/3}) )$ equals, for $\lambda_{k} \in \Bbb Q$, 
\begin{align*}
 \sum_{k=0}^{\infty}  \left ( \sum_{d \mid 3k+ \overline{2n+1}} \left ( \frac{d}{3} \right )d^{2n} \right )q^{\frac{3k+\overline{2n+1}}{3}} 
%
=  \sum_{k=0}^{2n+1} \lambda_{k}\Bigl (b^{3}(q) +c^{3}(q) \Bigr )^{\frac{2n+1 - \overline{2n+1} - 3k}{3}} c^{\overline{2n+1} +3k}(q).  
\end{align*}
Since  $m-\overline{m} \equiv 0 \pmod{3}$, each exponent of $(b^{3}(q)+c^{3}(q))$ in the last summand is a natural number. Therefore, the $\overline{2n+1}$-dissection of $f_{2n+1} (a(q^{1/3}), b(q^{1/3}) )$ is a homogeneous polynomial in $b(q)$ and $c(q)$ of degree at most $2n+1$. A similar calculation shows the $\overline{2n}$-dissection of the polynomial $g_{2n}(a(q^{1/3}), b(q^{1/3}) )$ is of the required form.
\end{proof}

The Hecke Eisenstein series of appropriate weight
may be expressed as polynomials in the cubic theta
functions by employing recursion formulas for cubic Eisenstein series derived by S. Cooper and others \cite{MR2249503}. 
To avoid a conflict of notation, we adopt the conventions
\begin{align}
  G_{2k,\mathbf{1}}(q) = \frac{L(1-2k,\mathbf{1})}{2} E_{2k,\mathbf{1}}&(q),\qquad G_{2k,\left ( \frac{\cdot}{3} \right )}(q) = \frac{L \left (-2k, \left ( \frac{\cdot}{3} \right ) \right )}{2} E_{2k+1,\left ( \frac{\cdot}{3} \right )}(q), \\ 
&G_{2k}(q) = \frac{\zeta(1 - 2k)}{2}E_{2k}(q).
\end{align}
\begin{lem} \label{cooper_rec}Let $\chi$ and $\mathbf{1}$ denote, respectively, the Jacobi symbol and the principal character modulo three. Then, for each integer $n\ge 1$, 
  \begin{align*}
    G_{2n+2,\chi}(q) = -9(2n&+1)(2n+2)G_{0,\chi}^{2}(q)G_{2n,\chi}(q) \\ &- 2(2n+1)(2n+2)\sum_{j=1}^{n-1} {2n \choose 2j} G_{2j,\chi}(q)G_{2n+2-2j}(q), \\ 
    G_{2n+2,\mathbf 1}(q) = 18 G_{0, \chi}&(q) G_{2n, \chi}(q) + 6 \sum_{j=1}^{n-1} {2n \choose 2j}  G_{2j,\chi}(q)G_{2n-2j,\chi}(q).
  \end{align*}
\end{lem}
By applying Lemma \ref{cooper_rec}, the cubic theta function
parameterizations for Eisenstein series of low weight
\cite{MR2249503,shen3}, as well as the familiar recursion formula
for Eisenstein series on the full modular group \cite[p. 30]{chandra}, we may generate each expansion displayed in the introduction and derive corresponding dissections for Eisenstein series of higher weight
\begin{align}
  \sum_{n=0}^{\infty}\Bigl ( &\sum_{d \mid 3 n +1} \left ( \frac{d}{3} \right ) d^{6} \Bigr )q^{n} = \frac{(q;q)_{\infty}^{17}}{(q^{3};q^{3})_{\infty}^3}+3^4 \cdot 50 q (q;q)_{\infty}^5 (q^{3};q^{3})_{\infty}^9+\frac{3^{9} \cdot 7 q^{2} (q^{3};q^{3})_{\infty}^{21}}{(q;q)_{\infty}^7}, \nonumber 
\end{align}
\begin{align}
 \sum_{n=0}^{\infty}\Bigl ( \sum_{d \mid 3 n +1} d^{9} \Bigr )q^{n} =& \frac{(q;q)_{\infty}^{26}}{(q^{3};q^{3})_{\infty}^6} + 3^{5} \cdot 23 \cdot 47 q (q;q)_{\infty}^{14} (q^{3};q^{3})_{\infty}^6 \\ &\qquad +3^{9} \cdot 2237 q^2 (q;q)_{\infty}^2 (q^{3};q^{3})_{\infty}^{18}+\frac{3^{13}\cdot 11\cdot 61 q^3 (q^{3};q^{3})_{\infty}^{30}}{(q;q)_{\infty}^{10}}, \nonumber 
\end{align}
\begin{align}
  \sum_{n=0}^{\infty}\Bigl ( \sum_{d \mid 3 n +2} \left ( \frac{d}{3} \right ) d^{10} \Bigr )q^{n} &= -\frac{3\cdot 341 (q;q)_{\infty}^{25}}{(q^{3};q^{3})_{\infty}^3}-3^{7}\cdot 4477 q (q;q)_{\infty}^{13} (q^{3};q^{3})_{\infty}^9 \\ &-3^{10} \cdot 20317 q^{2} (q;q)_{\infty} (q^{3};q^{3})_{\infty}^{21}-\frac{3^{15}\cdot 1847 q^{3} (q^{3};q^{3})_{\infty}^{33}}{(q;q)_{\infty}^{11}} \nonumber
\end{align}
\begin{align}
 \sum_{n=0}^{\infty}\Bigl (& \sum_{d \mid 3 n +1} \left ( \frac{d}{3}
 \right ) d^{12} \Bigr )q^{n} = \frac{(q;q)_{\infty}^{35}}{(q^{3};q^{3})_{\infty}^9} + 3^{4} \cdot 207076 q (q;q)_{\infty}^{23} (q^{3};q^{3})_{\infty}^3 \\ &+3^{9} \cdot 722810 q^{2} (q;q)_{\infty}^{11} (q^{3};q^{3})_{\infty}^{15}+\frac{3^{13} \cdot 722812 q^{3} (q^{3};q^{3})_{\infty}^{27}}{(q;q)_{\infty}}+\frac{3^{18} \cdot 55601 q^{4} (q^{3};q^{3})_{\infty}^{39}}{(q;q)_{\infty}^{13}}, \nonumber
\end{align}

\begin{align}
&\sum_{n=0}^{\infty}\Bigl ( \sum_{d \mid 3 n +2} d^{13} \Bigr )q^{n} = \frac{3 \cdot 2731 (q;q)_{\infty}^{34}}{(q^{3};q^{3})_{\infty}^6}+3^{6}\cdot 1674872 q (q;q)_{\infty}^{22} (q^{3};q^{3})_{\infty}^6 \\  &+3^{10} \cdot 9766130 q^{2} (q;q)_{\infty}^{10} (q^{3};q^{3})_{\infty}^{18}+\frac{3^{15} \cdot 2790064 q^{3} (q^{3};q^{3})_{\infty}^{30}}{(q;q)_{\infty}^2} +\frac{3^{19} \cdot 597871 q^{4} (q^{3};q^{3})_{\infty}^{42}}{(q;q)_{\infty}^{14}}, \nonumber 
\end{align}
\begin{align}
\sum_{n=0}^{\infty}\Bigl ( &\sum_{d \mid 3 n +1} d^{15} \Bigr )q^{n} = \frac{ (q;q)_{\infty}^{44}}{(q^{3};q^{3})_{\infty}^{12}} + 3^{4}\cdot 13\cdot 1019729 q (q;q)_{\infty}^{32} \\ &+3^{10} \cdot 80982274 q^2 (q;q)_{\infty}^{20} (q^{3};q^{3})_{\infty}^{12}+3^{14} \cdot 228972994 q^3 (q;q)_{\infty}^8 (q^{3};q^{3})_{\infty}^{24} \nonumber \\ &\qquad \qquad \qquad +\frac{3^{18} \cdot 152647045 q^4 (q^{3};q^{3})_{\infty}^{36}}{(q;q)_{\infty}^4}+\frac{3^{22}\cdot 28621321 q^5 (q^{3};q^{3})_{\infty}^{48}}{(q;q)_{\infty}^{16}}. \nonumber
\end{align}
To prove the congruences on line \eqref{eq:la}, apply Euler's theorem $a^{\varphi(n)} \equiv 1
\pmod{n}$, $\gcd (a,n)=1$, with $n = 3^{s+1}$ to reduce the divisor
sums to one of the identities of
Theorem \ref{garvans}. The congruences from \eqref{eq:la1},
follow from the fact that the sums from \eqref{eq:la1} reduce to those of
\eqref{eq:la} modulo $3^{s+1},$ since for each odd
prime $p$, \cite[Lemma 1]{MR0242764}
\begin{align}
  \label{eq:4star}
  n^{\frac{1}{2} \varphi(p^{\lambda})} =  \left ( \frac{n}{p} \right )
  \pmod{p^{\lambda}}, 
\quad p\nmid n,\quad \lambda \in \Bbb N.\end{align}

\section{Generalized cubic dissections}

We now consider cubic multisections of modular forms beyond the Eisenstein series studied in the previous section. Our purpose here is to study the dissection operators
\begin{align}
  \Omega_{3,k} \left ( \sum_{n=0}^{\infty} \nu_{n} q^{n} \right ) = \sum_{n=0}^{\infty} \nu_{3n+k} q^{n}, \qquad  \pi \left ( \sum_{n=0}^{\infty} \nu_{n} q^{n} \right ) = \sum_{n=0}^{\infty} \nu_{n} q^{n/3}
\end{align}
$k=0, 1, 2$, on subspaces of homogeneous polynomials in the cubic theta functions. We construct explicit matrix representations for these operators and discuss their spectral properties. Some of the integer eigenvalue and eigenvector pairs induce interesting congruences for the coefficients of corresponding eigenforms modulo powers of three. 
\begin{thm}\label{eq:31}
Let 
\begin{align}
  \frac{3^{5}}{504}(E_{6}(q^{3})-E_{6}(q)) = \sum_{n=0}^{\infty} u_{n} q^{n}, \qquad
27q(q^{3};q^{3})_{\infty}^{6}(q;q)_{\infty}^{6} =\sum_{n=0}^{\infty}v_{n}q^{n},
\end{align}
then 
\begin{align}
u_{3^{\ell n}} \equiv 0\ (mod\ 3^{5\ell}), \qquad v_{3^{\ell}n}\equiv0\ (mod\ 3^{2\ell}).  
\end{align}

\end{thm}
To prove the last theorem, we first derive expansions for more general operators.
\begin{thm} \label{th:gen}  let $\mathfrak{O}_{n}(x,y)$ denote the complex span of homogeneous polynomials in $x$ and $y$ of degree $d$ and let $B_{d}$ be the matrix whose $(r,k)$th entry equals
\begin{align} \label{whose}
\sum _{j=0}^k \sum _{\ell=0}^{3 d-3 j} \binom{k}{j} \binom{3 d-3 j}{\ell} \binom{3 j}{r-\ell}2^\ell  (-1)^{3 j+\ell-r} .
 \end{align}
Suppose that $f(q) \in \mathfrak{O}_{d}(a^{3}, c^{3})$ with
  \begin{align} \label{saw}
f(q)  =\sum_{n=0}^{\infty}v_{n}q^{n} =\sum_{k=0}^{d}\alpha_{k}a^{3(d-k)}(q)c^{3k}(q),    \qquad \alpha_{n}\in\mathbb{C},
  \end{align}
\vspace{-0.1in}
  \begin{align}
B_{d}(\alpha_{0},\alpha_{1},\ldots,\alpha_{d})^{T}=(\beta_{0},\beta_{1},\ldots,\beta_{d})^{T}.
  \end{align}
If $\mathbf{1}$ denotes the principal Dirichlet character
modulo three, then, for $m=0,1,2,$
\begin{align}
\Omega_{3,m}(f) = \sum_{n=0}^{\infty}k_{3n+m}q^{n}=q^{-m/3}\sum_{k=0}^{d-\mathbf{1}(m)}\beta_{3k+m}a^{3d-3k-m}(q)c^{3k+m}(q).
\end{align}
\end{thm}
The first few matrices $B_{d}$ for $d=1,2,3$ from Theorem \ref{th:gen} are given explicitly by
\begin{align}
\left(
\begin{array}{cccc}
 1 & 6 & 12 & 8 \\
 0 & 9 & 9 & 9 \\
\end{array}
\right)^{T}, \quad 
\left(
\begin{array}{ccccccc}
 1 & 12 & 60 & 160 & 240 & 192 & 64 \\
 0 & 9 & 63 & 171 & 234 & 180 & 72 \\
 0 & 0 & 81 & 162 & 243 & 162 & 81 \\
\end{array} 
\right)^{T},
\end{align}

\begin{align}
  \left(
\begin{array}{cccccccccc}
 1 & 18 & 144 & 672 & 2016 & 4032 & 5376 & 4608 & 2304 & 512 \\
 0 & 9 & 117 & 657 & 2088 & 4140 & 5328 & 4464 & 2304 & 576 \\
 0 & 0 & 81 & 648 & 2187 & 4212 & 5265 & 4374 & 2268 & 648 \\
 0 & 0 & 0 & 729 & 2187 & 4374 & 5103 & 4374 & 2187 & 729 \\
\end{array}
\right)^{T}.
\end{align}

\begin{proof}[Proof of \ref{th:gen}]
  Suppose that $f(q) \in \mathfrak{O}_{n}(a^{3}, c^{3})$ has has the expansions given by \eqref{saw}. Then, from the Borwein's identity \eqref{borwein} and the binomial theorem we deduce
  \begin{align}
   f(q)= \sum_{k=0}^{d}\alpha_{k}a^{3(d-k)}(a^{3} - b^{3})^{k} =  \sum_{k=0}^{d} \alpha_{k} \sum_{j=0}^{k} {k \choose j} (-1)^{3j}  a^{3d - 3 j} b^{3j}
  \end{align}
Applying Theorem \eqref{brv} and the binomial theorem, we see that $\pi(f) = f(q^{1/3})$ equals 
\begin{align} \nonumber
 &\sum_{k=0}^{d} \alpha_{k} \sum_{j=0}^{k} {k \choose j} (-1)^{3j}  (a+2c)^{3d -3j} (a-c)^{3j} \\ = 
& \sum_{k=0}^{d} \sum_{j=0}^{k} \sum_{\ell = 0}^{3 d - 3 j} \sum_{i=0}^{3j} {k \choose j} {3 d - 3 j \choose \ell} {3 j \choose i} (-1)^{3j - i} 2^{\ell} a^{3 d - \ell - i} c^{\ell + i}\alpha_{k}. \label{fg}
\end{align}
Therefore, for $0 \le k \le d$ and $0 \le r \le 3d$, the coefficient of $\alpha_{k} a^{3d - r}c^{r}$ equals the expression on line \eqref{whose}. Since the contribution to the $m$-dissection $\sum_{n=0}^{\infty}v_{n}q^{(3n+m)/3}$ of $\pi(f) =f(q^{1/3})$ arises entirely from terms in \eqref{fg} with $r\equiv m\ (mod\ 3)$ we have shown that the matrices defined by \eqref{whose} generate the requisite expansions for $\Omega_{3,m}(f)$. 
\end{proof}
From Theorem \ref{th:gen}, we deduce that the operator $\Omega_{3,0}$ over $\mathfrak{O}_{d}(a^{3}, c^{3})$ corresponds to a $(d+1)\times (d+1)$ matrix whose entries come from rows $1, 4, \ldots, 3d+1$ of $B_{d}$. In particular, the respective matrix representations for $\Omega_{3,0}$ on $\mathfrak{O}_{d}(a^{3}, c^{3})$, $d = 1, 2, 3$ are
\begin{align}
  \left(
\begin{array}{cc}
 1 & 0 \\
 8 & 9 \\
\end{array}
\right), \quad 
\left(
\begin{array}{ccc}
 1 & 0 & 0 \\
 160 & 171 & 162 \\
 64 & 72 & 81 \\
\end{array}
\right), \quad 
\left(
\begin{array}{cccc}
 1 & 0 & 0 & 0 \\
 672 & 657 & 648 & 729 \\
 5376 & 5328 & 5265 & 5103 \\
 512 & 576 & 648 & 729 \\
\end{array}
\right).
\end{align}
The eigenvectors $x_{i}$ and eigenvalues $\lambda_{i}$ for the matrix corresponding to $\mathfrak{O}_{2}(a^{3}, c^{3})$  are 
\begin{align}
x_{i} = \begin{pmatrix}
      0 & 9 & 4
\end{pmatrix}^{T}, \quad 
\begin{pmatrix}
0 & 1 & -1 
\end{pmatrix}^{T}, \quad 
\begin{pmatrix}
\ 121 & - 152 & 40
\end{pmatrix}^{T}, \quad \lambda = 243, 9, 1, \quad i=1,2,3. \nonumber
\end{align}
We may apply the cubic theta function parameterizations for Eisenstein series from \cite{MR2249503} to deduce that the eigenform for $x_{1}$ equals the Fourier expansion at $\tau = 0$ for the Hecke Eisenstein series corresponding to the principal character modulo three 
\begin{align} \label{lr}
 9a^{3}(q)c^{3}(q) + 4c^{6}(q)= \sum_{n=1}^{\infty} \frac{n^{5}(q^{n}+q^{2n})}{1-q^{3n}} = \frac{3^{5}}{504}(E_{6}(q^{3})-E_{6}(q)),
\end{align}
while $x_{2}$ and $x_{3}$ correspond, respectively, to the eigenforms 
\begin{align}
  121a^{6}(q) - 152 a^{3}(q)c^{3}(q) + 40c^{6}(q) = 121 E_{5, \chi_{3,1}}(q), \qquad a^{3}(q)c^{3}(q) - c^{6}(q)= c^{3}(q) b^{3}(q). \nonumber
\end{align}
This proves Theorem \ref{eq:31}. From divisor sum expansions for
\eqref{eisdef}, it follows that the Eisenstein series for primitive
Dirichlet character $\chi$ at the cusp $\tau=\infty$,
$E_{3j,\chi}(q)$, are eigenforms for $\Omega_{3,0}$ on
$\mathfrak{O}_{3j}(a^{3}, c^{3})$ with eigenvector $1$, while the
companion eigenforms
\begin{align*}
  \sum_{n=1}^{\infty}  \frac{n^{3j-1} (\chi(n) q^{n} + \chi(2) q^{2n})}{1-q^{3n}} = \sum_{n=1}^{\infty} \Bigl ( \sum_{d \mid n} \chi(n/d) d^{3j-1} \Bigr )q^{n}, \quad j \ge 1, \quad\chi(3j)=(-1)^{j},
\end{align*}
 corresponding to the Eisenstein series at the cusp $\tau=0$, have eigenvalue $3^{3j-1}$.

The matrices corresponding to $\Omega_{3,0}$ on
$\mathfrak{O}_{3j}(a^{3}, c^{3})$ have other interesting
properties paralleling those of corresponding quintic operators from
\cite{huber_jac}. We conjecture that, up to sign, the determinants are
powers of three. A proof of the conjecture may follow from
an analysis of the spectral structure of $\Omega_{3,0}$ by way of
classical Hecke operators. 
\begin{conj}
  Let $C_{d}$ denote the matrix representation for $\Omega_{3,0}$ on the vector space  $\mathfrak{O}_{3d}(a^{3}, c^{3})$ of homogeneous polynomials of degree $3d$ in $a^{3}(q)$ and $c^{3}(q)$ over $\Bbb C$. Then $\det C_{n} = \pm 3^{w(n)}$ for some $w(n) \in \Bbb N$. For even indices, we have $w(2n) = 4n(6n-2)$.
\end{conj}
Similar formulas may be obtained for the determinants of a class of
matrix representations for $\pi(f)$ corresponding to the embedding
$\pi(\mathfrak{O}_{n}(a, b))$ in $\mathfrak{O}_{n}(a,c)$. We give a precise construction for these trimidiation arrays in Theorem \ref{th:ag}. Our proof of the theorem
 is similar to that of Theorem \ref{th:gen}, and
depends primarily on the Borwein's identity \eqref{borwein}, the transformation formulas Theorem \ref{brv}, and the binomial theorem.
\begin{thm} \label{th:ag}
  Let $f(q)=\sum_{n=0}^{\infty}\nu_{n}q^{n}$ be a homogeneous polynomial in $a(q)$ and $b(q)$, 
  \begin{align}
f(q)=\sum_{n=0}^{d}\alpha_{n}a^{n}(q)b^{(d-n)}(q),    
  \end{align}
where $\alpha_{n}\in\mathbb{C}$. Then there exists a $(d+1)\times(d+1)$
  matrix $\mathcal{B}_{d}$ over $\mathbb{Z}$ such that if
  \begin{align}
    \mathcal{B}_{d}(\alpha_{0},\alpha_{1},\ldots,\alpha_{d})^{T}=(\beta_{0},\beta_{1},\ldots,\beta_{d})^{T}
  \end{align}
then
\begin{align}
\pi (f) = f(q^{1/3})=\sum_{n=0}^{d}\beta_{n}a^{n}(q)c^{(d-n)}(q),  
\end{align}
Moreover, the $(r,n)$th entry of $\mathcal{B}_{d}$ equals
\begin{align}
2^{d-n} (-1)^{n-r} \binom{n}{r} \, _2F_1\left(n-d,-r;n-r+1;-\frac{1}{2}\right).
\end{align}
\end{thm}
The hypergeometric function ${_{2}}F_{1}$, defined in \cite[p. 61]{aaa1}, arises from the the expansion 
\begin{align}
  \label{eq:m}
  \sum _{k=0}^{d-n} 2^{(d-n)-k} (-1)^{n-(r-k)} \binom{d-n}{k} \binom{n}{r-k}
\end{align}
for the entries of $\mathcal{B}_{d}$ as sums of binomial
coefficients. The first few matrices $\mathcal{B}_{d}$ defined in Theorem
\ref{th:ag} are
\begin{align*}
\mathcal{B}_{1}  =   \left(
\begin{array}{cc}
 2 & -1 \\
 1 & 1 \\
\end{array}
\right), \quad 
\mathcal{B}_{2}  =  \left(
\begin{array}{ccc}
 4 & -2 & 1 \\
 4 & 1 & -2 \\
 1 & 1 & 1 \\
\end{array}
\right), 
\quad
\mathcal{B}_{3}  =  \left(
\begin{array}{cccc}
 8 & -4 & 2 & -1 \\
 12 & 0 & -3 & 3 \\
 6 & 3 & 0 & -3 \\
 1 & 1 & 1 & 1 \\
\end{array}
\right). 
\end{align*}
As with the matrices in Theorem \ref{th:gen}, we observe that the determinants of the matrices from
Theorem \ref{th:ag} are certain powers of three. This leads to a more general conjecture. 
\begin{conj}
  Let $\mathcal{B}_{d}$ be the trimidiation matrix defined in Theorem \ref{th:ag}. Then $$\det \mathcal{B}_{n} = 3^{n(n+1)/2}.$$
\end{conj}

\end{document}